\documentclass[12pt]{article}
\usepackage{amsmath,amsthm}
\usepackage[pdftex]{graphicx}
\usepackage{amssymb}
\usepackage{amsfonts}
\usepackage{mathpazo}
\usepackage[all,cmtip]{xy}
\usepackage{color}
\usepackage{tikz}
\usetikzlibrary{matrix}
\usepackage{pdfpages}
\newtheorem{thm}{Theorem}[section]
\newtheorem{theorem}[thm]{Theorem}

\newtheorem{proposition}[thm]{Proposition}

\newtheorem{definition}[thm]{Definition}

\newtheorem{observation}[thm]{Observation}

\title{Rationality of the universal \\ K3 surface of genus 8}
\author{Daniele Di Tullio}

\begin{document}

\maketitle

\begin{abstract}
The aim of the present paper is to prove the rationality of the universal family of polarized $ K3 $ surfaces of degree 14. This is achieved by identifying it with the moduli space of cubic fourfolds plus the data of a quartic scroll. The last moduli space is finally proved to be rational since it has a natural structure of $\mathbb P^n$-bundle over a $ k $-stably rational variety with $k \leq n$.
\end{abstract}
\section{Moduli of polarized K3 surfaces}
The properties of the 19-dimensional moduli spaces $ \mathcal{F}_g $ of polarized $ K3 $ surfaces of genus $ g $ (equivalently of degree $ 2g - 2 $), parametrizing up to isomorphism pairs $ (S,H) $, where $ S $ is a $ K3 $ surface and $ H\in \operatorname{Pic}(S) $ is a big and nef class satisfying $ {H^{2} = 2g - 2} $, have been deeply studied in the recent years as well as the universal families
\[
\mathcal{F}_{g,k}\to\mathcal{F}_{g}
\]
which parametrize up to isomorphism triples $ (S,P,H) $, where $ P\in  S^{[k]}$ (it is an Hilbert Subscheme of $ S $ of dimension 0 and length $ k $). An interesting
problem is to know what type of varieties they are.  Gritsenko, Hulek and Sankaran in \cite{GHS} proved that $ \mathcal{F}_{g} $ is a variety of general type for $ g > 62 $ and for $ {g = 47,51,53,55,58,59,61} $. On the other side $\mathcal{F}_{g, 1}  \to\mathcal{F}_{g} \color{black}$ is never of general type. Indeed this morphism is fibered in Calaby-Yau varieties hence, by Iiyaka's easy addition formula, its Kodaira dimension is bound by $\dim \mathcal F_g = 19$.  Then $\mathcal F_{g,1}$ cannot be of general type, cfr. \cite[5.4]{K3 genus 14}. 
Mukai, in a celebrated series of papers,  proved the unirationality of $ \mathcal{F}_{g,1} $ for $ g \leq 12 $, and for $ g = 13,16,18,20 $.  See  \cite{M1},\cite{M2},\cite{M3},\cite{M4}],\cite{M5}. Recently Farkas and Verra proved in \cite{K3 genus 14} that the universal $ K3 $ surface $ \mathcal{F}_{14,1} $ is rational and in 
\cite{K3 genus 22} the unirationality of $ \mathcal{F}_{22,1} $.

The main result of the present paper is
\begin{theorem}
	$ \mathcal{F}_{8,1} $ is a rational variety.
	\label{Rationality U K3}
\end{theorem}
The proof consists of two steps
\begin{itemize}
	\item[1)] To prove that $ \mathcal{F}_{8,1} $ is birational to $ \mathbb{P}^{16}\times \operatorname{Pic}_{3,2} $;
	\item[2)] To show that $ \mathbb{P}^{1}\times\operatorname{Pic}_{3,2} $ is rational, so a fortiori $ \mathcal{F}_{8,1} $ is.
\end{itemize}
%%%%%%%%%%%%%%%%%%%%%%%%%%%%%%%%%%NEW SECTION***********************************
\section{Moduli of cubic fourfolds and Hassett divisors.}
Cubic fourfolds are hypersurfaces of degree 3 in a 5-dimensional projective space. Their moduli space admits a rather intuitive description: it is the GIT quotient $ \mathbb{P}(H^{0}(\mathcal{O}_{\mathbb{P}^{5}}(3)))/\hspace{-4pt}/\operatorname{PGL}(6) $, see \cite{MFK} and \cite{Laza 2009} for the characterization of the set of semistable points for this group action. A less immediate way to parametrize them is via period mappings and Hodge theory. Claire Voisin established in \cite{Voisin} a global Torelli theorem for cubic fourfolds, so introducing  the power of Hodge theory in the study of their moduli space. Other properties of the period map have been studied by Radu Laza, Eduard Looijenga in \cite{Laza 2010} and \cite{Looijenga}. 

In \cite{Hasset} Hassett used Hodge theoretic methods to construct certain divisors in the moduli of cubic fourfolds: it is known that for general cubic fourfold $ X\in \mathcal{C} $ there are no primitive algebraic 2-cycles.
algebraic 2-cycles, i.e. $H^{2,2}(X, \mathbb Z) = \mathbb{Z}h^{2}$. Notice also that the homological and numerical equivalences on $X$ implies that $H^{2,2}(X, \mathbb Z)$ is
the N\`eron-Severi lattice $\operatorname{NS}^2(X)$ of codimension $2$ algebraic cycles modulo numerical equivalence.
Hassett showed that the condition of containing an unexpected algebraic 2-cycle is divisorial in $ \mathcal{C} $. These cubic fourfolds are called "special" and some countable union of  their divisorial families is expected to be the closure of the set of all rational examples. Furthermore every divisor parametrizing an irreducible family of special cubic fourfolds $X$ is characterized as follows. For a general $X$ in the family the lattice $\operatorname{NS}^2(X)$ has rank two and in this case it is isomorphic to a fixed, positive definite rank two lattice $\mathbb L_d$ of discriminant $d$, generated by $h^2$ and by a further exceptional class. In particular each irreducible component is labeled by some $\mathbb L_d$, see \cite{Hasset} and \cite{Laza 2010}. 
An important construction made in \cite{Hasset} is that of a polarized K3 surface $(S,H)$, which appears as associated to $X$ if $X$ belongs to the above mentioned countable union. Then the orthogonal complement of the embedding of $\mathbb L_d$ in $H^{2,2}(X, \mathbb Z)$ as $\operatorname{NS}^2(X)$ is isomorphic to the primitive Hodge structure of a polarized K3 surface $(S,H)$ such that $H^2 = d$. In particular we have $d = 2g-2$, where $g$ is the so called genus of $(S,H)$. Hassett shows that the isomorphism exists if $d$ is an admissible number ($ d=2(n^2+n+1),n\geq 2 $). In particular the number $2g-2$ is admissible if $ g=n^{2}+n+2 $ . In these very special cases it happens that the Fano 
variety $F(X)$ of the lines in $X$ is birational to the Hilbert scheme $S^{[2]}$ of two points of $S$. Moreover one can recover
$F(X)$ and then $X$ from $S^{[2]}$. Hassett shows that this defines a natural rational map
\[
\mathcal{F}_{n^2+n+2}\to \mathcal{C}_{2(n^2+n+1)}
\]
which is birational for $ n \equiv 0,2\mod{3} $ and of degree $ 2 $ for $ n \equiv 1  ( \mod3)  $. In particular for $ n=2 $ there is a birational morphism
\[
\mathcal{F}_{8}\to \mathcal{C}_{14}.
\]

For  the initial values of $ d  = 2g-2 $ there is a more concrete description of the $ K3 $ surface associated to a special cubic fourfold $ X\in \mathcal{C}_{d} $.To see it we previously introduce some notation. Let us consider an integral
rational scroll $R \subset \mathbb P^5$. We assume that ${\rm Sing}(R)$ consists of finitely many \it non normal nodes \rm $o$. This means by definition that $o$ has multiplicity $2$, the normalization of $R$ is smooth  and the inverse image of $o$ is a set of two points. For simplicity we introduce the following:
\begin{definition}  $R$ is a \it rational nodal scroll. We say that $R$ is a $\delta$-nodal rational scroll if $\vert \operatorname{Sing}(R) \vert = \delta$.
\end{definition}
Denoting by an overline the topological closure in $\mathcal C$, \color{black} we have that:
\begin{itemize}
	\item $ \mathcal{C}_{14}:= \overline{\{\text{Cubic fourfolds containing a smooth  quartic scroll}\}}$;
	\item $ \mathcal{C}_{26}:=\overline{\{\text{Cubic fourfolds containing a 3-nodal  rational scroll of degree $7$} \}} $;
	\item $ \mathcal{C}_{42}:=\overline{\{\text{Cubic fourfolds containing an 8-nodal  rational scroll of degree 9}\}} $.
\end{itemize}
For these values of $ d $ the associated $ K3 $ surface to a general  $ X\in \mathcal{C}_{d} $ is isomorphic to the  Hilbert Scheme of the corresponding scrolls contained in $ X $. Setting \color{black}
\[
\tilde{\mathcal{C}}_{14}:=\{(X,R):X\in \vert \mathcal{O}_{\mathbb{P}^{5}}(3)\vert, R\subset X\text{ is a quartic scroll} \} /\hspace{-4pt}/ \operatorname{PGL}(6)
\]
we can therefore  lift the birational morphism $ \mathcal{F}_{8}\to \mathcal{C}_{14} $
\begin{center}
	\begin{tikzpicture}
	\node (1) at (0,0) {$\tilde{\mathcal{C}}_{14}$};
	\node (2) at (2,0) {$\mathcal{F}_{8,1}$};
	\node (3) at (0,-1) {$\mathcal{C}_{14}$};
	\node (4) at (2,-1) {$\mathcal{F}_{8}.$};
	\path[->]
	(1) 	edge 	(2)
	(1)  edge  (3)
	(2)	edge (4)
	(3) edge (4);
	\end{tikzpicture}
\end{center}
So Theorem \ref{Rationality U K3} can be restated
\begin{theorem}
	$\tilde{\mathcal{C}}_{14}$ is a rational variety.
\end{theorem}
\section{$ \tilde{\mathcal{C}}_{14} $ as a projective bundle over $ \operatorname{Pic}_{3,2} $}
The main idea of the proof is to show that $ \tilde{\mathcal{C}}_{14}$ has a structure of projective bundle over ${\operatorname{Pic}_{3,2}} $. Here we use the notation $ \operatorname{Pic}_{d,g} $ to denote the universal
Picard variety of line bundles of degree $ d $ over a curve of genus $ g $. In
what follows this is the coarse moduli space of pairs $ (C, \mathcal{L}) $ such that $ C $ is a
smooth integral curve of genus $ g $ and $ \mathcal{L} \in \operatorname{Pic}_{d}(C) $, see \cite{Harris Morrison} for the main general properties and definitions. In our case a birationally equivalent construction of it as a GIT quotient can be provided as follows.
Observe that
a pair $ (C, \mathcal{L}) $, defining a general point of $ \operatorname{Pic}_{3,2} $, provides an embedding
\[
C\hookrightarrow \mathbb{P}^{1}\times\mathbb{P}^{1}
\]
as a curve of type $ (3, 2) $ such that $ \mathcal{L}\cong \mathcal{O}_{C}(0,1) $ and $ \omega_{C}\cong \mathcal{O}_{C}{(1,0)} $. Then a
birationally equivalent model of $ \operatorname{Pic}_{3,2} $ is the GIT quotient
\[
\left|	\mathcal{O}_{\mathbb{P}^{1}\times\mathbb{P}^{1}}(3,2)	\right|/\hspace{-4pt}/\operatorname{Aut}(\mathbb{P}^{1})^{2}
\]
Let $ R \subset \mathbb{P}^{5}$ be a quartic scroll.  The moduli space $ \tilde{\mathcal{C}}_{14} $ is also described as a quotient
$
\mathfrak{C}_{R}/G_{R}
$
where \[ \mathfrak{C}_{R}:=\left\lbrace
X\in \left|\mathcal{O}_{\mathbb{P}^{5}}(3) \right|
\text{ s.t. }X\supset R
\right\rbrace\text{  and  }G_{R}:={\{f\in \operatorname{PGL}(5)\text{ s.t. }f(R)=R\}}\]
This is clear since every quartic scroll can be moved to a fixed one by an automorphism of  $ \mathbb{P}^{5} $ .
The first step is to construct a projective bundle
\[
\mathfrak{C}_{R}\to \left| \mathcal{O}_{\mathbb{P}^{1}\times\mathbb{P}^{1}}(3,2)\right|
\]
We recall that a Segre product is the embedding of $ \mathbb{P}^{a}\times\mathbb{P}^{b} $ defined by the line bundle $ \mathcal{O}_{\mathbb{P}^{1}\times \mathbb{P}^{2}}(1,1) $, up to projective automorphisms. The degree of this embedding is $ {a+b \choose a} $. Clearly a Segre product of degree 3 is $ \mathbb{P}^{1}\times \mathbb{P}^{2} $ embedded in $ \mathbb{P}^{5} $. For convenience in the exposition we will say that
\begin{definition}
	A cubic Segre product is an embedding $ \mathbb{P}^{1}\times\mathbb{P}^{2}\hookrightarrow \mathbb{P}^{5} $ as above.
\end{definition}
Note that $ \vert \mathcal{O}_{\mathbb{P}^{1}\times\mathbb{P}^{2}}(1,0)\vert $ is the pencil of two by two disjoint planes
\[
\{x\}\times \mathbb{P}^{2},x\in\mathbb{P}^{1}.
\]
On the other hand the elements of $ \vert \mathcal{O}_{\mathbb{P}^{1}\times \mathbb{P}^{2}} \vert$ are smooth quadric surfaces, namely products $ \mathbb{P}^{1}\times \mathcal{L} $ with $ \mathcal{L}\in \mathcal{O}_{\mathbb{P}^{2}}(1) $. We are interested to smooth quadric surfaces of degree 4 in $ \mathbb{P}^{1}\times\mathbb{P}^{2} $. To this purpose let us point out that the surfaces of degree 4 in $ \mathbb{P}^{1}\times\mathbb{P}^{2} $ are distributed in two linear systems:
\begin{itemize}
	\item  $ |\mathcal{O}_{\mathbb{P}^{1}\times\mathbb{P}^{2}}(0,2)| $;
	\item $ \left| \mathcal{O}_{\mathbb{P}^{1}\times \mathbb{P}^{2}}(2,1) \right| $.
\end{itemize}
The next propositions describe smooth quartic scrolls from these linear systems. We fix the usual notation $ \mathbb{F}_{n} $ for the $ \mathbb{P}^{1} $-bundle over $ \mathbb{P}^{1} $ with minimal section $ e $ of self intersection  $ -n $. We denote the fibre of $ \mathbb{F}_{n}\to \mathbb{P}^{1} $ by $ f $. As is well known $ \mathbb{F}_{n} $ is a Hirzebruch surface. We are interested to $ \mathbb{F}_{0}=\mathbb{P}^{1}\times \mathbb{P}^{1} $ and to $ \mathbb{F}_{2} $ which is a rank 3 quadric cone blown up at its vertex.
\begin{proposition}
	Let $ R\in \left| \mathcal{O}_{\mathbb{P}^{1}\times\mathbb{P}^{2}}(0,2) \right| $ be a smooth and irreducible divisor. Then $ R \cong \mathbb{F}_{0} $ and
	\[
	\mathcal{O}_{R}(1,0)\cong\mathcal{O}_{R}(e),\;\; \mathcal{O}_{R}(0,1)\cong \mathcal{O}_{R}(2f).
	\]
	Moreover a unique cubic Segre product contains $ R $ as an element of $ |\mathcal{O}_{\mathbb{P}^{1}\times\mathbb{P}^{2}}(0,2)| $.
\end{proposition}
\begin{proof}
	Let $ p_{2}:\mathbb{P}^{1}\times \mathbb{P}^{2}\to \mathbb{P}^{2} $ the second projection map. Since $ |R|= p_{2}^{*}|\mathcal{O}_{\mathbb{P}^{2}}(2)| $, then $ R=\mathbb{P}^{1}\times B $, where $ B\subset \mathbb{P}^{2} $ is a smooth conic. Moreover $ \mathbb{P}^{1}\times \mathbb{P}^{2} $ is the union of the planes $ \{x\}\times\mathbb{P}^{2}, x \in \mathbb{P}^{1} $, and $ \{x\}\times\mathbb{P}^{2} $ is spanned by the conic $ \{x\}\times B \subset R$. Hence $ \mathbb{P}^{1}\times\mathbb{P}^{2} $ is uniquely associated to $ R $.
\end{proof}
\begin{proposition}
	Let $ R \in \left| \mathcal{O}_{\mathbb{P}^{1}\times\mathbb{P}^{2}}(2,1) \right| $ be a smooth and irreducible divisor. Then $ R $ is $ \mathbb{F}_{n} $ with $ n\in \{0,2\} $ and
	\[
	\mathcal{O}_{R}(1,0)\cong \mathcal{O}_{R}(f),\;\; \mathcal{O}_{R}(0,1)\cong \mathcal{O}_{R}\left( \frac{n+2}{2}f+e \right).
	\]
	The family of cubic Segre products containing $ R $ as an element of $ |\mathcal{O}_{\mathbb{P}^{1}\times \mathbb{P}^{2}}(2,1)| $ is naturally parametrized by an open set of $ |\mathcal{O}_{R}(0,1)|^{*} $.
\end{proposition}
\begin{proof}
	We have that $ p_{2}|_{R} :R\to\mathbb{P}^{2}$ is a generically finite morphism of degree 2. Fixing bi-homogeneous coordinates $ [Z_{0}:Z_{1}]\times [X_{0}:X_{1}:X_{2}] $ on $ \mathbb{P}^{1}\times \mathbb{P}^{2} $, the equation of $ R $ is $ aZ_{0}^{2} + b Z_{0}Z_{1} +cZ_{1}^{2}=0 $, where $ a,b,c $ are linear forms in $ [X_{0}:X_{1}:X_{2}] $. Since $ R $ is smooth, then the branch curve $ B $ of $ p_{2}|_{R} $ is a conic of rank $ \geq 2 $. If $ B $ is smooth then $ p_{2}|_{R} $ is finite and $ R\cong \mathbb{F}_{0} $. If $ B $ has rank $ 2 $, then $ R $ is the blowing up of a quadric cone in its singular point, so it is $ \mathbb{F}_{2} $. Now observe that $ V _1:={p_{1}}|_{R}^{*}H^{0}(\mathcal{O}_{\mathbb{P}^{1}}(1))=H^{0}(\mathcal{O}_{R}(f)) $ and that $ \mathcal{L}:=p_{2}|_{R}^{*}\mathcal{O}_{\mathbb{P}^{2}}(1) $ is the line bundle defining the model of $ R $ as a quadric surface. in particular $ V_{2}:=p_{2}|_{R}^{*}H^{0}(\mathcal{O}_{\mathbb{P}^{2}}(1)) $ has codimension 1 in $ H^{0}(\mathcal{L}) $. Finally consider
	\begin{center}
		\begin{tikzpicture}
		\node (1) at (0,0) {$	H^{0}(\mathcal{O}_{\mathbb{P}^{1}\times\mathbb{P}^{2}}(1,0))\otimes
			H^{0}(\mathcal{O}_{\mathbb{P}^{1}\times\mathbb{P}^{2}}(0,1))$};
		\node (2) at (5,0) {$V_{1}\otimes V_{2}$};
		\node (3) at (7.5,0) {$H^{0}(\mathcal{O}_{R}(1))$};
		\path[->]
		(1) 	edge node[above]{$ r $}	(2)
		(2)  edge node[above]{$m$} (3);
		\end{tikzpicture}
	\end{center}
where $ r $ is the restriction and $ m $ is the multiplication map. It is standard to check that both $ r $ and $ m $ are isomorphisms. This implies that $ V_{2} $ uniquely reconstructs $ \mathbb{P}^{1}\times\mathbb{P}^{2} $ and that the family of cubic Segre products containing $ R $ as an element of $ |\mathcal{O}_{\mathbb{P}^{1}\times\mathbb{P}^{2}}(2,1)| $ is birationally parametrized by $ |\mathcal{O}_{R}(0,1)|^{*} $.
\end{proof}
Let $ R $ be a smooth quartic scroll, that is a Hirzebruch surface $ \mathbb{F}_{n} $ with $ n\in \{0,2\} $. Keeping in account the previous propositions and their proofs, it is easy to associate to $ R $ a union of planes $ T $ containing $ R $ as follows. Consider the pencil $ \left| f+\frac{n}{2}e \right| $ and the union of planes
\[
T:=\bigcup_{c\in\left| f+\frac{n}{2}e \right|}T_{c}
\]
where $ T_{c} $ is the plane spanned by $ c $. Notice that $ c $ is a smooth conic if $ n=0 $ and the rank 2 conic $ e+f',f'\in|f|$ if $ n=2 $.

\begin{theorem}
	$ R $ is in a unique cubic Segre product if $ n=0 $ and in a unique cone of vertex $ e $ over a rational normal cubic if $ n=2 $.
\end{theorem}
In what follows it will be enough to assume $ n=0 $. We have a chain of embeddings
\begin{center}
	\begin{tikzpicture}
	\node (1) at (0,0) {$\mathbb{P}^{1}\times\mathbb{P}^{1}$};
	\node (2) at (4,0) {$\mathbb{P}^{1}\times\mathbb{P}^{2}$};
	\node (3) at (7.5,0) {$\mathbb{P}^{5}$};
	\path[->]
	(1)  edge node[above]{$\left(\begin{smallmatrix}
		\operatorname{id}_{\mathbb{P}^{1}} & 0\\
		0 & \left| \mathcal{O}(2) \right|
		\end{smallmatrix}\right)$} (2)
	(2) edge node[above]{$|\mathcal{O}(1,1)|$} (3);
	\end{tikzpicture}
\end{center}
Let $F\in H^{0}\left( \mathcal{O}_{\mathbb{P}^{1}\times \mathbb{P}^{2}}(0,2) \right) $ the defining polynomial of $ R $ in $ T $. Then the restriction map
\[
\pi:H^{0}\left(\mathcal{O}_{\mathbb{P}^{5}}(3) \right)\to 
H^{0}\left(	\mathcal{O}_{\mathbb{P}^{1}\times\mathbb{P}^{2}}(3,3)	\right)
\]
has the property that $ \pi(H^{0}(\mathcal{I}_{R}(3)))=f\cdot H^{0}(\mathcal{O}_{\mathbb{P}^{1}\times\mathbb{P}^{2}}(3,1)) $.
The restriction
\[
H^{0}(\mathcal{O}_{\mathbb{P}^{1}\times\mathbb{P}^{2}}(3,1))\to
H^{0}(\mathcal{O}_{\mathbb{P}^{1}\times\mathbb{P}^{1}}(3,2))
\]
is an isomorphism of vector spaces. Consequently there is an induced homomorphism
\[
H^{0}(\mathcal{I}_{R}(3))\to H^{0}( \mathcal{O}_{\mathbb{P^{1}\times\mathbb{P}^{1}}}(3,2) )
\]
which induces a linear projection.
	\[
\mathfrak{C}_{R}\to \left|\mathcal{O}_{\mathbb{P}^{1}\times\mathbb{P}^{1}}(3,2) \right|
\]
We want this map to descend to a $ \mathbb{P}^{16} $-bundle
\[
\tilde{\mathcal{C}}_{14}\to \operatorname{Pic}_{3,2}
\]
Recall some general fact about GIT.
\begin{definition}
	Let $ X $ be and algebraic variety and let $ \mathcal{F} $ be a coherent sheaf on the semistable locus $ X^{ss}$. Let $ G $ be a reductive algebraic group acting on $ X $. Then $ \mathcal{F} $ is said to descend to $ X/\hspace{-4pt}/G $ if there is a coherent sheaf $ \overline{\mathcal{F}} $ on $ X/\hspace{-4pt}/G $ whose pullback under the quotient map $ X^{ss}\to X/\hspace{-4pt}/G $ is the original sheaf $ \mathcal{F} $.
\end{definition}
If $ \mathcal{F} $ is a vector bundle the following result gives a criteria for the descending:
\begin{theorem}[Kempf]
Let $ X $ be a quasi-projective scheme over an algebraically closed field $\kappa $ of characteristic zero, and let $ G $ be a reductive algebraic group defined over $ \kappa $ which acts on $ X $ with a fixed choice of  linearization $ H $. Let $ E $ be a $ G $-vector bundle on $ X^{ss} $. Then $ E $ 
descends to $ X/\hspace{-4pt}/G $ if and only if for every closed point $ x $ of $ X^{ss} $ such that the orbit $ G\cdot x $ is closed in $ X^{ss} $, the stabilizer of $ x $ in $ G $ acts trivially on the fiber $ E_{x} $ of $ E $ at $ x $.
\label{Kempf}
\end{theorem}
\begin{proof}
	See \cite{Drezet-Narasimhan}.
\end{proof}
Recall also a standard result (\cite[12.9]{Hartshorne}):
\begin{theorem}[Grauert]
	Let $ f:X\to Y $ be a projective morphism of noetherian schemes, and let $ \mathcal{F} $ be a coherent sheaf on $ X $, flat over $ Y $. Suppose furthermore that $ Y $ is integral and that for some $ i $, the function $ h^{i}(X_y,\mathcal{F}_y)  $ is constant on $ Y $. Then $ R^{i}\mathcal{F} $ is locally free on $ Y $ and the natural map 
	\[
	R^{i}f_{*}(\mathcal{F})\otimes k(y)\to H^{i}(X_{y},\mathcal{F}_{y})
	\]
	is an isomorphism.
\end{theorem}
The Kempf and Grauert theorems are the key ingredients in the proof of the next proposition. A similar argument was used by Shepherd-Barron in \cite[6]{Shepherd-Barron} to prove the rationality of $ \mathcal{M}_{6} $.
\begin{proposition}
	$ \tilde{\mathcal{C}}_{14} $ is birational to $ \mathbb{P}^{16}\times \mathbb{Pic}_{3,2} $.
\end{proposition}
\begin{proof}
The rational map	
	\[
	{\mathfrak{C}}_{R}\to |\mathcal{O}_{\mathbb{P}^{1}\times\mathbb{P}^{1}}{(3,2)}|
	\]
	is a linear projection
	\[
	\pi:\mathbb{P}(V)\to \mathbb{P}(V')
	\]
	where $ V:=H^{0}(\mathcal{I}_{R}(3)), V':=H^{0}(\mathcal{O}_{\mathbb{P}^{1}\times\mathbb{P}^{1}}(3,2))$. Let $ \alpha:\tilde{\mathbb{P}} \to \mathbb{P}(V)$ be the blow-up a of the base locus of the projection and $ {\tilde{\pi}:=\alpha\circ\pi} $. Let $ {\mathcal{L} :=\alpha^{*}\mathcal{O}_{\mathbb{P}(V)}(1)}$. We have that $ \operatorname{PGL}(2)^{2} $ acts freely on a open subvariety of $ \mathbb{P}(V') $, so $ \mathcal{L} $ is $ \operatorname{PGL}(2) $-linearized. It follows from Kempf theorem that it descends to a line bundle on $ \mathbb{P}(V)/\hspace{-4pt}/\operatorname{PGL}(2)^{2} $. Furthermore it restricts to $ \mathcal{O}(1) $ on the fibers of the map
	\[
	\tilde{\mathbb{P}}\to \mathbb{P}(V').
	\]
	From Grauert theorem we have that $ \tilde{\pi}_{*}\mathcal{L} $ is locally free on $ \mathbb{P}(V') $ and the projective fibration $ \alpha $ is isomorphic to $ \mathbb{P}(\tilde{pi}_{*}\mathcal{L}) $. It follows that $ \alpha $ is a projective bundle.
\end{proof}

\section{Stable Rationality of $ \operatorname{Pic}_{32} $}
In this section we show that the projectivization of the pull-back on $ \operatorname{Pic}_{32}$ of the Hodge bundle over $ \mathcal{M}_{2} $ is a rational variety. It follows that $ \operatorname{Pic}_{32} \times \mathbb{P}^{1}$ is rational. Then a fortiori also $ \tilde{\mathcal{C}}_{14} $ is rational. Recall that $ \mathcal{M}_{g} $ is endowed with a sheaf named \textit{Hodge bundle}. Over a suitable non empty open set the Hodge bundle is a rank $ g $ vector bundle $ \Lambda_{g} $ with fibre $ H^0(\omega_C) $ at the moduli point of C, see e.g. \cite{Loo1} or \cite{Harris Morrison}.
\begin{definition}
	The Hodge bundle over $ \operatorname{Pic}_{3,2} $ is the rank-2 vector bundle
	\[
	\Lambda_{3,2}\to\operatorname{Pic}_{3,2}
	\]
	defined as the pullback of $ \Lambda_{2}\to\mathcal{M}_{2} $ under the natural map $ \operatorname{Pic}_{3,2}\to\mathcal{M}_{2} $
\end{definition}

Let $ \mathbb{K}_{g} $ be the projectivization of $ \Lambda_{g} $. Then $ \mathbb{K}_{g} $ fits in the general theory of moduli of abelian differentials, see \cite{FP}. Indeed an open set of it is the coarse moduli space of couples $ (C,K) $ such that $ C $ is a smooth, connected genus $ g $ curve and $ K $ is a smooth canonical divisor of $ C $.
\begin{definition}$ \mathbb{K}_{3,2} $ is the pull-back of $ \mathbb{K}_{2} $ through the natural map $ \operatorname{Pic}_{3,2}\to \mathcal{M}_{2} $.
\end{definition}

In particular it follows that $ \mathbb{K}_{3,2} $ represents the coarse moduli space of triples $ (C,\mathcal{L},K) $. let us consider the projection map
\[
p:\mathbb{K}_{3,2}\to\operatorname{Pic}_{3,2}
\]
then $ p $ is a $ \mathbb{P}^{1} $-bundle over an open set of $ \operatorname{Pic}_{3,2} $. Its fibre over the moduli point of $ (C,\mathcal{L}) $ in $ \operatorname{Pic}_{3,2} $ is $ |\omega_{C}| $. Now it is not difficult to construct a useful family of triples $ (C,\mathcal{L},K) $ dominating $ \mathbb{K}_{3,2} $ via the moduli map.

For a general triple $ (C,\mathcal{L},K) $ we can assume that $ \mathcal{L} $ is globally generated and that $ K $ consists of two distinct points $ K=o_{1} + o_{2}$ with $ o_{1}\neq o_{2} $. Let $ p:C\to\mathbb{P}^{1} $ and $ q:=C\to\mathbb{P}^{1} $ be the morphisms respectively defined by $ \omega_{C} $ and $ \mathcal{L} $. Then $ p\times q $ defines an embedding
\[
C\subset\mathbb{P}^{1}\times\mathbb{P}^{1}
\]
with two marked points, that are the images of $ o_{1} $ and $ o_{2} $. With some abuse of notation, we still denote them by $ o_{1},o_{2} $. In particular $ C $ is a smooth element of the linear system $ |\mathcal{O}_{\mathbb{P}^{1}\times\mathbb{P}^{1}}(3,2)| $ and contains $ \{o_{1},o_{2}\} $. et $ \mathcal{I} $ be the ideal sheaf of $ \{o_{1},o_{2}\} $ in $ \mathbb{P}^{1}\times\mathbb{P}^{1} $ and let
\begin{center}
	\begin{tikzpicture}
	\node (1) at (0,0) {$|\mathcal{I}(3,2)|$};
	\node (2) at (3,0) {$\mathbb{K}_{3,2}$};
	\node (3) at (0,-0.5) {$C$};
	\node (4) at (3,-0.5) {$(C,\mathcal{O}_{C}(0,1),o_{1}+o_{2})$};
	\path[->]
	(1) 	edge node[above]{$m$}	(2);
	\path[|->]
	(3) edge (4);
	\end{tikzpicture}
\end{center}
be the natural moduli map. Since $ (C,\mathcal{L},K) $ defines a general point of $ \mathbb{K}_{3,2} $ the next property is immediate.
\begin{proposition}
	$ m:|\mathcal{I}(3,2)|\dashrightarrow \mathbb{K}_{3,2} $ is dominant.
\end{proposition}
The following result gives a more concrete description of $ \mathbb{K}_{3,2} $.
\begin{proposition}
	Let $ m:|\mathcal{I}(3,2)|\dashrightarrow \mathbb{K}_{3,2} $ as above. Then
	\[
		m(C)=m(C') \iff \exists\sigma\in\operatorname{Stab}_{\operatorname{PGL}(2)^2}(\{o_{1},o_{2}\})\text{ s.t. }\sigma(C)=C'
	\]
\end{proposition}
\begin{proof}
	The isomorphisms $\sigma:C\to C' $ which are restrictions of an element of $ \operatorname{Aut(\mathbb{P}^{1})}^{2} $ are determined by the conditions
	\[
	\sigma^{*}\mathcal{O}_{C'}(1,0)=\mathcal{O}_{C}(1,0),\;\;\sigma^{*}\mathcal{O}_{C'}(0,1)=\sigma^{*}\mathcal{O}_{C}(0,1)
	\]
	Imposing that $ \sigma^*(o_{1}+o_{2})=o_{1}+o_{2}$ means exactly that $\sigma\in\operatorname{Stab}_{\operatorname{PGL}(2)^2}(\{o_{1},o_{2}\})$.
\end{proof}
\begin{observation}
Since it is possible to move any $\{o_{1},o_{2}\}\subset D\in |\mathcal{O}_{\mathbb{P}^{1} \times \mathbb{P}^{1}}(1,0)|$ to $ {\{([1:0],[1:0]),([1:0],[0:1])\}} $ through an element of $ \operatorname{PGL}(2)^{2} $, $ \mathbb{K}_{3,2} $ can be described as the quotient of
\[
\{
C\in 
	\left| 
	\mathcal{O}_{\mathbb{P}^{1}\times\mathbb{P}^{1}}(3,2) 
	\right|
: ([1:0],[0:1]),([1:0],[1:0])\in C
\}
\]
modulo the equivalence relation
\begin{eqnarray*}
&C\sim C'&\\
& \iff &\\ &\exists\sigma=(\sigma_{1},\sigma_{2}):\sigma(C)=C,\sigma_{1}([1:0])=[1:0],&\\
&\sigma_{2}(\{[1:0],[0:1]\})=\{ [1:0],[0:1] \}&
\end{eqnarray*}
\label{observation relative canonical}
\end{observation}
We use the characterization of the observation \ref{observation relative canonical} to prove its rationality using an argument of classical GIT.
\begin{proposition}
	$\mathbb{K}_{3,2} $ is birational to $ \mathbb{P}^{6} $.
\end{proposition}
\begin{proof}
From the observation \ref{observation relative canonical} $ \mathbb{K}_{3,2}$ can be described as the GIT quotient set of $ (3,2) $-divisors
\[
\left\lbrace
\begin{matrix}
C_{011}X_{0}^{3}Y_{0}Y_{1} + X_{0}^2X_{1}(C_{120}Y_{0}^2+C_{111}Y_{0}Y_{1}+C_{102}Y_{1}^{2})+\\
+X_{0}X_{1}^2(C_{220}Y_{0}^{2}+C_{211}Y_{0}Y_{1}+C_{202}Y_{1}^2)+\\
+X_{1}^3(C_{320}Y_{0}^2+C_{311}Y_{0}Y_{1}+C_{302}Y_{1}^2)=0
\end{matrix}
\right\rbrace
\]
modulo the action of the group $ G\subset \operatorname{PGL}(2)^{2} $ defined by
\[
G:=\operatorname{Stab}(\{o_{1},o_{2}\})=\left\lbrace
\left(
\begin{bmatrix}
* & *\\
0 & *\\
\end{bmatrix},
\begin{bmatrix}
* & 0\\
0 & *\\
\end{bmatrix}
\right)
\right\rbrace
\cup
\left\lbrace
\left(
\begin{bmatrix}
* & *\\
0 & *\\
\end{bmatrix},
\begin{bmatrix}
0 & *\\
* & 0\\
\end{bmatrix}
\right)
\right\rbrace
\]
where $ o_{1}:=([1:0],[1:0]) $, $ o_{2}:=([1:0],[0:1]) $.
For a general element of this set we have that $ C_{011}\neq 0 $. It follows that in the $ G $-orbit of a general element there is an element with $ C_{111}=0 $: just make the transformation $ X_{0}\mapsto X_{0} + \frac{C_{111}}{3C_{011}}X_{1} $. Note that if two elements with $C_{111}=0  $ are in the same $ G $-orbit, then they are necessarily connected by an element of ${ G' \leq G}$, where
\[
G':=\left\lbrace
\left(
\begin{bmatrix}
* & 0\\
0 & *\\
\end{bmatrix},
\begin{bmatrix}
* & 0\\
0 & *\\
\end{bmatrix}
\right)
\right\rbrace
\cup
\left\lbrace
\left(
\begin{bmatrix}
* & 0\\
0 & *\\
\end{bmatrix},
\begin{bmatrix}
0 & *\\
* & 0\\
\end{bmatrix}
\right)
\right\rbrace
\]
and vice versa $ G' $ acts on set of polynomials with $ C_{111}=0 $. So another birational model of $ \mathbb{K}_{3,2} $ is the GIT quotient 
\[
\left\lbrace
\begin{matrix}
C_{011}X_{0}^{3}Y_{0}Y_{1} + X_{0}^2X_{1}(C_{120}Y_{0}^2+C_{102}Y_{1}^{2})+\\
+X_{0}X_{1}^2(C_{220}Y_{0}^{2}+C_{211}Y_{0}Y_{1}+C_{202}Y_{1}^2)+\\
+X_{1}^3(C_{320}Y_{0}^2+C_{311}Y_{0}Y_{1}+C_{302}Y_{1}^2)=0
\end{matrix}
\middle\rbrace \middle/\hspace{-5pt} \right/G'
\]
The same space can be described as the GIT quotient
\[
\left\lbrace
\begin{matrix}
C_{011}Y_{0}Y_{1} + x(C_{120}Y_{0}^2+C_{102}Y_{1}^{2})+\\
+x^2(C_{220}Y_{0}^{2}+C_{211}Y_{0}Y_{1}+C_{202}Y_{1}^2)+\\
+x^3(C_{320}Y_{0}^2+C_{311}Y_{0}Y_{1}+C_{302}Y_{1}^2)
\end{matrix}
\middle\rbrace\middle/\hspace{-5pt}\right/ H
\]
where $ H:={\mathbb{C}^{*}}^{3}\rtimes (\mathbb{Z}/2\mathbb{Z}) $. $ H $ acts in the following way:
\begin{itemize}
	\item ${\mathbb{C}^{*}}^{3} $ acts by multiplying the variables $ x,Y_{0},Y_{1} $ by costants, more precisely:
	\[
		(a,b,c)\cdot (C_{ijk})=(a^{i}b^{j}c^{k}C_{ijk});
	\]
	\item $ \mathbb{Z}/2\mathbb{Z} $ inverts the variables $ Y_{0} $ and $ Y_{1} $, more precisely
	\[
	(1\mod 2) \cdot (C_{ijk})=C_{ikj}.
	\]
\end{itemize}
We find now 6 invariants which completely determines the GIT quotient of the dense open subset
 \[
\left\lbrace
\begin{matrix}
C_{011}\ne 0,C_{120}\ne 0,C_{102}\ne 0,C_{220}\ne 0,C_{211}\ne 0,\\
C_{202}\ne 0,C_{320}\ne 0,C_{311}\ne 0,C_{302}\ne 0 \end{matrix} 
\middle\rbrace\middle/\hspace{-5pt}\right/ H
\]
then giving a birational map to $ \mathbb{C}^{6} $. We first compute the invariants for the action of $ {\mathbb{C}^{*}}^{3}$, then we recover the invariants for the action of $ H $. We use that
\[
X/\hspace{-4pt}/G\cong (X/\hspace{-4pt}/N)/\hspace{-4pt}/(G/N)
\]
in our case $ G=H, N={\mathbb{C}^{*}}^{3}, G/N=\mathbb{Z}/2\mathbb{Z}$.
The action of $ {\mathbb{C}^{*}}^{3} $ is uniquely determined by the following invariants
\begin{eqnarray*}
&(I_{1},J_{2},J_{3},I_{4},I_{5},I_{6})&\\
&=&\\
&\left(
\dfrac{C_{120}C_{102}}{C_{211}C_{011}},\dfrac{C_{011}C_{311}}{C_{220}C_{102}},
\dfrac{C_{302}C_{011}}{C_{211}C_{102}},\dfrac{C_{220}C_{202}}{C_{211}^{2}},
\dfrac{C_{320}C_{302}}{C_{311}^{2}},\dfrac{C_{211}^{3}}{C_{311}^{2}C_{011}}
\right)&
\end{eqnarray*}
It is immediate that they are invariants, we need to check that they uniquely determine an isomorphism class. Suppose that two polynomials (identified with elements of $ {\mathbb{C}^{*}}^{9} $)\\
$ {(C_{011},C_{120},C_{102},C_{220},C_{211},C_{202},C_{320},C_{311},C_{302})} $ and\\
$ {(C_{011}',C_{120}',C_{102}',C_{220}',C_{211}',C_{202}',C_{320}',C_{311}',C_{302}')} $ have the property that
\begin{eqnarray}
(I_{1},J_{2},J_{3},I_{4},I_{5},I_{6})=
(I_{1}',J_{2}',J_{3}',I_{4}',I_{5}',I_{6}')
\label{eq1}
\end{eqnarray}
Acting by $ {\mathbb{C}^{*}}^3 $ on both of them it is possible to make
\begin{eqnarray}
{C_{211}=C_{211}'=1}, &{C_{311}=C_{311}'=1}, &{C_{102}=C_{102}'=1}
\label{eq2}
\end{eqnarray}
in fact it is just needed to choose two triples $ (a,b,c) $ and $ (a',b',c') $ such that
\[
\begin{cases}
a^{2}bc=C_{211}^{-1}\\
a^{3}bc=C_{311}^{-1}\\
ac^{2}=C_{102}^{-1}
\end{cases}\;\;
\begin{cases}
{a'}^{2}b'c'=C_{211}'^{-1}\\
{a'}^{3}b'c'=C_{311}'^{-1}\\
{a'}c'^{2}=C_{102}'^{-1}
\end{cases}\;\;
\]
If the equalities \ref{eq1} and \ref{eq2} hold, then $ C_{ijjk}=C_{ijk}' $, in fact:
\begin{enumerate}
	\item $ {I_{6}=I_{6}'}\iff \dfrac{C_{211}^{3}}{C_{311}^{2}C_{011}}=\dfrac{C_{211}'^{3}}{C_{31}'^{2}C_{011}'}\Rightarrow C_{011}=C_{011}'$;
	\item $ J_{2}=J_{2}'\iff \dfrac{C_{011}C_{311}}{C_{220}C_{102}}=\dfrac{C_{011}'C_{311}'}{C_{220}'C_{102}'}\Rightarrow C_{220}=C_{220}' $;
	\item $ {I_{4}=I_{4}'} \iff \dfrac{C_{220}C_{202}}{C_{211}^{2}}=\dfrac{C_{220}'C_{202}'}{C_{211}'^{2}}\Rightarrow C_{202}=C_{202}'$;
	\item $ {I_{1}=I_{1}'} \iff \dfrac{C_{120}C_{102}}{C_{211}C_{011}}=\dfrac{C_{120}'C_{102}'}{C_{211}'C_{011}'}\Rightarrow C_{120}=C_{120}'$;
	\item  $ {J_{3}=J_{3}'}\iff \dfrac{C_{302}C_{011}}{C_{211}C_{102}}=\dfrac{C_{302}'C_{011}'}{C_{211}'C_{102}'}\Rightarrow C_{302}=C_{320}' $;
	\item $ {I_{5}=I_{5}'} \iff \dfrac{C_{320}C_{302}}{C_{311}^{2}}=\dfrac{C_{320}'C_{302}'}{C_{311}'^{2}}\Rightarrow C_{320}=C_{320}'$.
\end{enumerate}
Note that the invariants $ J_{2},J_{3} $ are not invariant for the action of $ \mathbb{Z}/2\mathbb{Z} $. In fact let $ \iota $ be the involution exchanging the variables $ Y_{0} $ and $ Y_{1} $, then\\ 
\[ \iota(J_{2})=\dfrac{C_{011}C_{311}}{C_{202}C_{120}}=I_{1}^{-1}{J_{2}}^{-1}I_{4}^{-1}I_{6}^{-1}\]
\[ \iota(J_{3})=\dfrac{C_{320}C_{011}}{C_{211}C_{120}}=I_{1}^{-1}{J_{3}}^{-1}I_{5}I_{6}^{-1}\]
It is standard to check that the action of $ \mathbb{Z}/2\mathbb{Z} $ on $ {\mathbb{C}^{*}}^{6} $ given by
\[
\iota(x_{1},x_{2},x_{3},x_{4},x_{5},x_{6})=(x_{1},x_{1}^{-1}x_{2}^{-1}x_{4}^{-1}x_{6}^{-1},x_{1}^{-1}x_{3}^{-1}x_{5}x_{6}^{-1},x_{4},x_{5},x_{6})
\]
is uniquely determined by the invariants
\[
(x_{1},x_{2}+x_{1}^{-1}x_{2}^{-1}x_{4}^{-1}x_{6}^{-1},x_{3}+x_{1}^{-1}x_{3}^{-1}x_{5}x_{6}^{-1},x_{4},x_{5},x_{6})
\]
It follows that the invariants of the action of $ H $ are
\begin{eqnarray*}
	&(I_{1},I_{2},I_{3},I_{4},I_{5},I_{6})&\\
	&=&\\
	&\left( 
	\frac{C_{120}C_{102}}{C_{211}C_{011}},
	\frac{C_{011}C_{311}}{C_{220}C_{102}}+\frac{C_{011}C_{311}}{C_{202}C_{120}},
	\frac{C_{302}C_{011}}{C_{211}C_{102}}+\frac{C_{320}C_{011}}{C_{211}C_{120}},
	\frac{C_{220}C_{202}}{C_{211}^{2}},
	\frac{C_{320}C_{302}}{C_{311}^{2}},\frac{C_{211}^{3}}{C_{311}^{2}C_{011}}
	\right)&
\end{eqnarray*}
which gives a birational correspondence between $ \mathbb{K}_{3,2}$ and $ \mathbb{C}^{6} $.
\end{proof}
\section{Acknowledgements}
This work is part of the author’s PhD thesis under the supervision of Alessandro Verra. The author is very grateful to him for suggesting the problem and some possible strategies and for the accurate review.


\begin{thebibliography}{99}
	
	\bibitem[Dol]{Dolgachev Invariant theory} Dolgachev, I. (2003) Lectures on Invariant Theory. CUP.
	
	\bibitem[DN]{Drezet-Narasimhan}J.-M. Drezet and M. S. Narasimhan.
	 Groupe de Picard des vari\'{e}t\'{e}s de modules de  fibr\'{e}s semi-stables sur les courbes alg\'{e}briques.  \underline{Invent. Math.}, 97(1):53–94, 1989.
	 
	 \bibitem[FP]{FP} G. Farkas and R. Pandharipande, The moduli space of twisted canonical divisors, ArXiv 1508.07940 2015.
	 
	 \bibitem[FV1]{K3 genus 14} G. Farkas and A. Verra, The universal K3 surface of genus 14 via cubic fourfolds. \underline{J. de Math\'{e}matiques Pures et Appliqu\'{e}es} 111 (2018), 120.
	 
	 \bibitem[FV2]{K3 genus 22} G. Farkas and A. Verra, The unirationality of the moduli space of $ K3 $ surfaces of genus 22, arXiv preprint (2019), 117.
	
	 	 
	 \bibitem[GHS]{GHS} V. Gritsenko, K. Hulek and G.K. Sankaran, The Kodaira dimension of the moduli space of K3 surfaces. \underline{Inventiones Mathematicae} 169 (2007), 519–567.
	 
	 \bibitem[Har]{Hartshorne} R. Hartshorne, Algebraic Geometry. \underline{Springer Graduate Texts in Mathematics} 52 (1977).
	 
	 \bibitem[Has1]{Hasset}B. Hassett, Special cubic fourfolds. \underline{Comp. Math.} 120 (2000), no. 1, 1–23. 
	 
	 \bibitem[HM]{Harris Morrison} 
	  J. Harris and I. Morrison, Moduli of Curves, \underline{Graduate Texts in Math.} 187 (1998), Springer New York.
	 
	 \bibitem[Laz1]{Laza 2009} R. Laza. The moduli space of cubic fourfolds, \underline{J. Algebraic Geom.} 18 (2009), no. 3, 511–545. 
	 
	 \bibitem[Laz2]{Laza 2010} R. Laza. The moduli space of cubic fourfolds via the period map. \underline{Ann. of Math.} (2), 172(1):673–711, 2010.
	 
	 \bibitem[Loo1]{Loo1} E. Looijenga. A minicourse on moduli of curves. School on Algebraic
	 Geometry (Trieste, 1999), 267291, \underline{ICTP Lect.} Notes, 1, Abdus Salam
	 Int. Cent. Theoret. Phys., Trieste, 2000.
	 
	 \bibitem[Loo2]{Looijenga} Eduard Looijenga. The period map for cubic fourfolds. \underline{Invent. Math.}, 177(1):213–233, 2009.
	 
	 \bibitem[MFK]{MFK} D. Mumford, J. Fogarty, F. Kirwan. Geometric Invariant Theory.
	 	\underline{Ergebnisse der Mathematik und ihrer Grenzgebiete. 2. Folge}, 34. Springer-Verlag Berlin Heidelberg, 1994.
 	 \bibitem[Mu1]{M1} S. Mukai, Curves, $ K3 $ surfaces and Fano 3-folds of genus $ \leq 10 $, in: \underline{Algebraic Geometry and Commutative Algebra in Honor of M. Nagata}, 357-377, Kinokuniya, Tokyo, 1988.
	 
	 \bibitem[Mu2]{M2} S. Mukai, Curves and symmetric spaces II,\underline{ Annals of Mathematics} 172 (2010), 1539-1558. 
	 
	 \bibitem[Mu3]{M3} S. Mukai. Polarized K3 surfaces of genus 18 and 20, in: \underline{Complex Projective Geometry, London Math. Soc. Lecture Notes Series} 179, 264-276, Cambridge University Press 1992. 
	 
	 \bibitem[Mu4]{M4} S. Mukai. Polarised K3 surfaces of genus 13, in: \underline{Moduli Spaces and Arithmetic Geometry} (Kyoto 2004), Advanced Studies in Pure Mathematics 45, 315-326, 2006. 
	 
	 \bibitem[Mu5]{M5} S. Mukai, K3 surfaces of genus 16, RIMS preprint 1743, 2012, available at
	http:/\hspace{-4pt}/www.kurims.kyotou.ac.jp/preprint/file/RIMS1743.pdf.
	
	\bibitem[She]{Shepherd-Barron} N.I. Shepherd-Barron, Invariant theory for $ S_{5} $ and the rationality of $ M_6 $. \underline{Compositio Mathematica} 70 (1989),3-25.
	
	\bibitem[Vo]{Voisin} Claire Voisin, Th\'{e}or\`{e}me de Torelli pour les cubiques de $ \mathbb{P}^{5} $. \underline{Invent. Math.} 86 (1986), no. 3, 577–601.
\end{thebibliography}
\end{document}